\documentclass[english]{article}
\usepackage[T1]{fontenc}
\usepackage[latin9]{inputenc}
\usepackage{amsmath}
\usepackage{amssymb}
\usepackage{tikz}

\makeatletter
\usepackage{amsfonts}
\usepackage{amsmath}

\usepackage{amsfonts}\usepackage{amsthm}\usepackage{color}\usepackage{enumitem}\usepackage{bm}\usepackage{cancel}\usepackage{thmtools}

\date{}

\def\theenumi{\arabic{enumi}}

\def\theenumii{\alph{enumii}}
\def\p@enumii{\theenumi.}

\def\theenumiii{\arabic{enumiii}}
\def\p@enumiii{(\theenumi)(\theenumii)}

\def\p@enumiv{\p@enumiii.\theenumiii}
\newtheorem{theorem}{Theorem}[section]

\newtheorem{claim}[theorem]{Claim}

\newtheorem{corollary}[theorem]{Corollary}

\newtheorem{fact}[theorem]{Fact}
\newtheorem{lemma}[theorem]{Lemma}

\newtheorem{proposition}[theorem]{Proposition}

\theoremstyle{definition}
\newtheorem{definition}[theorem]{Definition}

\newtheorem{remark}[theorem]{Remark}

\makeatother

\usepackage{babel}
\begin{document}

\title{Shalom's property $H_{\mathrm{FD}}$ and extensions by $\mathbb{Z}$ of locally finite groups}

\author{J\'er\'emie Brieussel and Tianyi Zheng}
\maketitle
\begin{abstract}
We show that every finitely generated extension by $\mathbb{Z}$ of a locally normally finite group has Shalom's property $H_{\mathrm{FD}}$. This is no longer true without the normality assumption. This permits to answer some questions of Shalom, Erschler-Ozawa and Kozma. We also obtain a Neumann-Neumann embedding result that any countable locally finite group embedds into a two generated amenable group with property $H_{\mathrm{FD}}$.
\end{abstract}

\section{Introduction}

A finitely generated group has property $H_{\mathrm{FD}}$ (resp. $H_{\mathrm{T}}$) if every unitary representation with non-trivial reduced first cohomology admits a finite dimensional (resp. trivial) subrepresentation. Property $H_{\mathrm{FD}}$ was introduced by Shalom \cite{Shalom} as an invariant of quasi-isometry among finitely generated amenable groups. Recently, by showing directly that a group of polynomial growth has Property $H_\mathrm{FD}$, Ozawa \cite{Ozawa} gave a functional analysis proof of Gromov's theorem on groups of polynomial growth.

Amenable groups with property $H_{\mathrm{FD}}$ are somewhat ``small'' groups. In \cite{Shalom}, various families of solvable groups having, and not having property $H_{\mathrm{FD}}$ were shown. For instance by \cite{Shalom}, polycyclic groups and wreath products $(\mathbb{Z}/n\mathbb{Z}) \wr \mathbb{Z}$, $n\in\mathbb{N}$, have property $H_{\mathrm{FD}}$, while $\mathbb{Z} \wr \mathbb{Z}$ does not. Towards a unified explanation for these examples, Shalom conjectured in \cite[Section 6]{Shalom} that for solvable groups, property $H_{\mathrm{FD}}$ would be equivalent to finite Hirsch length (see \cite{Hillman} for the definition of  the Hirsch length of  elementary amenable groups). We describe counter-examples to both directions of this conjecture. In one direction, in Proposition \ref{wreath3} we show that the wreath product $(\mathbb{Z}/2\mathbb{Z})\wr \mathbb{Z}^d$ with $d\ge 3$ does not have property $H_{\mathrm{FD}}$, while it has finite Hirsch length. In the other direction, we construct a solvable group of infinite Hirsch length which has property $H_{\mathrm{FD}}$ by applying a result of Gournay \cite{Gournay}, see  Proposition \ref{3solv}. 

Aiming to generalize the case of lamplighter groups, Shalom also asked whether all locally-finite-by-$\mathbb{Z}$ groups have $H_{\mathrm{FD}}$. This is not true in general, in subsection \ref{sec:noHFD} we construct explicit harmonic cocycles with weakly mixing representation to show the following examples of locally-finite-by-$\mathbb{Z}$ groups don't have property $H_{\rm{FD}}$: the group $\textrm{Sym}(\mathbb{Z}) \rtimes \mathbb{Z}$, the discrete affine group of a regular tree and some locally nilpotent by $\mathbb{Z}$ groups introduced by Gromov \cite{Gromov} --- see Propositions \ref{SymZ}, \ref{DA} and \ref{prop_Gromov}. However, we prove that any locally-normally-finite-by-$\mathbb{Z}$ group has Shalom's property $H_{\mathrm{FD}}$.

Recall that a group is locally finite if any finite subset is included in a finite subgroup. We say a group is {\it locally normally finite} if any finite subset is included in a finite normal subgroup. For instance, the group $\oplus_{\mathbb{Z}} \mathbb{Z}/2\mathbb{Z}$ is locally normally finite, whereas the group $\textrm{Sym}(\mathbb{Z})$ of finitely supported permutations is locally finite but not locally normally finite.

\begin{theorem}\label{main}
Let $G$ be a finitely generated group that fits in an exact sequence
\[
1\to L\to G\to\mathbb{Z}\to1.
\]
where $L$ is locally normally finite. Then the group $G$ has Shalom's property $H_{\mathrm{T}}$, hence $H_{\mathrm{FD}}$. 
\end{theorem}

We remark that in Theorem \ref{main}, the assumption that $L$ is locally normally finite is an "intrinsic" property of the group $L$, it doesn't depend on the extension. Although algebraically the locally-normally-finite assumption is restrictive,
random walks on the class of locally-normally-finite-by-$\mathbb{Z}$ groups  exhibit very rich behavior. See examples and more details in subsection \ref{subsec: example}.
In particular, we obtain new examples of groups with property $H_{\rm{FD}}$.

Erschler and Ozawa  \cite{EO} have given a criterion for the property that any finite dimensional summand of $b$ is cohomologically trivial,  where $b$ is a $1$-cocycle with coefficients in a unitary representation.
As a corollary, they obtain a sufficient condition for $H_{\mathrm{FD}}$: if a group admits a symmetric probability measure $\mu$ with finite generating support such that $\limsup_n \mu^{(n)}(B(e,c\sqrt{n}))>0$ for all $c>0$, then it has property $H_{\mathrm{FD}}$. They pointed out that all previously known examples of amenable groups with $H_{\mathrm{FD}}$ were covered by this corollary. As an application of Theorem \ref{main}, we obtain examples where the speed of random walk can be super-diffusive.

\begin{corollary}\label{cor}
For any function $f:[1,\infty) \to [1,\infty)$ such that $f(1)=1$ and $\frac{f(x)}{\sqrt{x}}$ is  non-decreasing and $\frac{f(x)}{x}$ decreases to $0$ as $x\to \infty$, there exists a group with property $H_{\mathrm{FD}}$ and a symmetric probability measure $\mu$ of finite generating support with speed of random walk $\mathbb{E}_{\mu^{(n)}}d(e,x)\simeq f(n)$.
\end{corollary}

Here $g(n)\simeq f(n)$ means there is $C\ge1$ with $\frac{f(n)}{C}\le g(n) \le Cf(n)$. Corollary \ref{cor} follows from Theorem \ref{main} as the groups introduced in \cite{BZ} with prescribed speed function are locally-normaly-finite-by-$\mathbb{Z}$, see \cite[Fact 2.10]{BZ}. These groups can also be chosen to have prescribed $\ell^p$-isoperimetry or $L_p$-compression. 

In fact, examples of groups with property $H_{\mathrm{FD}}$ not satisfying the assumption of Erschler-Ozawa's corollary mentioned above can be obtained without Theorem \ref{main} using a result of Gournay \cite{Gournay}, who proved that harmonic cocycles of weakly mixing representations factorize by the FC-center $Z^{\mathrm{FC}}$, which consists of elements with finite conjugacy classes. Indeed, the construction of \cite[Section 2]{BZ} can be modified by taking finite factor groups. The group obtained is an FC-central extension by a lamplighter group, which has property $H_{\mathrm{T}}$ by \cite{Shalom}. See more details in Subsection \ref{subsec: example}. This variation of construction provides us groups with the following properties.

 \begin{corollary}\label{cor-function}
Let $f:\mathbb{R}_{+} \to \mathbb{R}_{+}$ be a sub-addtive function such that $\frac{f(x)}{x}\to 0$ as $x\to \infty$. Then there is a finite generated group $\Delta$ and a symmetric probability measure $\mu$ of finite generating support on $\Delta$ such that for some constant $c>0$, for all $n\ge 1$,
$$\mathbb{E}_{\mu^{(n)}}d(e,x) \ge cf(n);$$
while all sub-exponential growth $\mu$-harmonic functions on $\Delta$ factors through a quotient map $\Delta \to F\wr \mathbb{Z}$ with $F$ finite. 
In particular, all sublinear $\mu$-harmonic functions on $\Delta$ are constant.
\end{corollary}

A folklore conjecture of Kozma asks if diffusive speed of $\mu$-random walk on a group $G$  is equivalent to absense of non-constant sublinear $\mu$-harmonic functions. The examples of Corollary \ref{cor-function} answer one direction negatively. Because of currently limited understanding of groups with diffusive simple random walk behavior, the other direction of the conjecture (if diffusive speed implies all sublinear harmonic functions are constant) remains open.

In the spirit of the classical Neumann-Neumann embedding \cite{Neumann2}, we show that any countable locally finite group embeds into a group with Shalom's property $H_{\mathrm{FD}}$.

\begin{proposition}\label{NN}
Let $H$ be a countable locally finite group. Then there exists a
two generated infinite amenable group $G$ with Shalom's
property $H_{\mathrm{FD}}$ and containing $H$ as an embedded subgroup.
\end{proposition}

The proof uses a sufficient
condition for $H_{\mathrm{FD}}$ from Erschler-Ozawa \cite{EO}.

We wish to point out that it is not known whether the lamplighter group $\mathbb{Z}/2\mathbb{Z} \wr \mathbb{Z}^2$ has property $H_{\mathrm{FD}}$. Neither is known an example of a non-Liouville amenable group with property $H_{\mathrm{FD}}$.

The paper is organized as follows. Definitions and notations are given in Section \ref{sec:setting}. Theorem \ref{main} is proved in Subsection \ref{subsec:main_proof} and Corollary \ref{cor-function} is proved in Subsection \ref{subsec: example}. It is proved in Section \ref{wrZd} that $(\mathbb{Z}/2\mathbb{Z})\wr \mathbb{Z}^d$ do not have property $H_{\textrm{FD}}$ when $d \ge 3$. Examples of locally-finite-by-$\mathbb{Z}$ groups without property $H_{\mathrm{FD}}$ are given in Section \ref{sec:noHFD}. Proposition \ref{NN} is derived in Section \ref{NNemb}. In Appendix \ref{appendix}, we prove that the wreath product $H\wr G$ of two infinite finitely generated amenable groups does not have property $H_{\mathrm{FD}}$. This answers another question of Shalom \cite[Section 6]{Shalom}, who had proved it when $H$ has infinite abelianization.

\section{Shalom's property and harmonic cocycles}\label{sec:setting}

We recall definitions and basic facts about Shalom's property $H_{\mathrm{FD}}$. More details can be found in \cite{Shalom}, \cite{Ozawa}, \cite{EO}.

Let $G$ be a finitely generated group and $\pi:G \to \mathcal{U}(\mathcal{H})$ a unitary representation into a Hilbert space $\mathcal{H}$. A function $b:G\to \mathcal{H}$ is a $1$-cocycle if $b(gh)=b(g)+\pi_gb(h)$ for all $g,h$ in $G$, and a $1$-coboundary if there exists $v$ in $\mathcal{H}$ such that $b(g)=v-\pi_gv$ for all $g$ in $G$. Denote $Z^1(G,\pi)$ the vector space of $1$-cocycles and $B^1(G,\pi)$ the subspace of $1$-coboundaries.

The first reduced cohomology of $\pi$ is the quotient space $$\overline{H}^1(G,\pi)=Z^1(G,\pi)/\overline{B^1(G,\pi)},$$ where the closure is taken with respect to uniform convergence on compact (i.e. finite) subsets.

\begin{definition}[Shalom \cite{Shalom}]
A group $G$ has property $H_{\mathrm{FD}}$ (resp. $H_{\mathrm{F}}$, $H_{\mathrm{T}}$) if any unitary representation $\pi$ with $\overline{H^1}(G,\pi) \neq 0$ admits a subrepresentation which is finite dimensional (resp. finite, trivial).
\end{definition}

Given a symmetric finitely supported non-degenerate probability measure $\mu$ on $G$, the scalar product 
\[
\langle b,b'\rangle_{\mu}=\sum_{g\in G} \langle b(g),b'(g) \rangle \mu(g)
\] 
gives a Hilbert space structure on $Z^1(G,\pi)$. The associated topology coincides with uniform convergence on compact subsets.

Harmonic $1$-cocycles were implicitly introduced by Guichardet in \cite {Guich}, where it was observed that every element in the first reduced cohomology $\overline{H}^1(G,\pi)$ is uniquely represented by a $\mu$-harmonic cocycle. A cocycle $b$ is $\mu$-harmonic if 
\[
\sum_{g \in G} b(xg)\mu(g)=b(x) \textrm{ for all }x\in G,
\]
 or equivalently $\sum_{g \in G} b(g)\mu(g)=0$. This happens if and only if $b$ is orthogonal to the space of $1$-coboundaries. So there is a vector space isomorphism $$\overline{H}^1(G,\pi) \simeq \overline{B}^1(G,\pi)^\perp.$$

It follows that a group $G$ has property $H_{\mathrm{FD}}$ if and only if any $\mu$-harmonic cocycle of a weakly mixing representation (i.e. a representation without finite dimensionnal invariant subspace) is zero.

\section{Kernel of harmonic cocycles }
Throughout the rest of the paper, we always assume that $\mu$ is a symmetric probability measure with finite generating support on $G$.

\subsection{Extensions by $\mathbb{Z}$ of locally normally finite groups}\label{subsec:main_proof}

\begin{proposition}\label{main_prop}

Let $G$ be a finitely generated group that fits in an exact sequence
\[
1\to N\to G\to\mathbb{Z}\to1.
\]
Suppose $F$ is a finite normal subgroup of $N$ and the projection of $\mu$ to $\mathbb{Z}$ is the law of a simple random walk.
Let $b:G\to\mathcal{H}$ be a $\mu$-harmonic $1$-cocycle with weakly
mixing linear part $\pi$. Then 
\[
\forall \gamma \in F, \quad b(\gamma)=0.
\]


\end{proposition}

The proof of Proposition \ref{main_prop} follows the standard coupling argument as in \cite[Theorem 1]{BDCKY}, 
which shows that the lamplighter group $F\wr \mathbb{Z}$ with $F$ finite does not admit non-constant sublinear harmonic functions.
The following lemma adapted from \cite[Proposition]{Ozawa} provides information on the growth of a harmonic cocycle 
with weakly mixing representation.
It roughly says that a harmonic cocycle with weakly mixing representation grows slower than $n^{1/2}$ 
on the set that doesn't translate $\mu^{(n)}$ too far. 

\begin{lemma}[{\cite[Proposition]{Ozawa}}]\label{byOzawa}

Let $b:G\to\mathcal{H}$ be a $\mu$-harmonic cocycle with 
weakly mixing representation $\pi$. 
Let $0<\delta<\frac{1}{3}$ be a constant.
Define 
\[
\mathcal{S}_{n}(\delta):=\left\{g\in G:\ \frac{1}{2}\left\Vert \mu^{(n)}-g\mu^{(n)}\right\Vert _{1}<\delta\right\}.
\]
Then 
\[
\frac{1}{\sqrt{n}}\sup_{g\in\mathcal{S}_{n}(\delta)}\left\Vert b(g)\right\Vert _{\mathcal{H}}\longrightarrow0\ \mbox{as }n\to\infty.
\]

\end{lemma}

\begin{proof}
Since the formulation is different from the original statement in \cite{Ozawa}, we include a proof here for completeness.
By \cite[Lemma]{Ozawa}, 
\[
\sup_{\xi,\left\Vert \xi\right\Vert =1}\frac{1}{n}\sum_{x}\mu^{(n)}(x)\left|\left\langle b(x),\xi\right\rangle \right|^{2}\longrightarrow 0.
\]
It follows that given $\delta$, for any $\epsilon>0$, there exists
a constant $N$ such that for all $n\ge N$,$ $
\[
\sup_{\xi,\left\Vert \xi\right\Vert =1}\sum_{x}\mu^{(n)}(x)\left|\left\langle b(x),\xi\right\rangle \right|^{2}\le\epsilon^{2}\delta n.
\]

For any unit vector $\xi$, let 
\[
E_{\xi}=\left\{ x\in G:\ \left|\left\langle b(x),\xi\right\rangle \right|\le\epsilon n^{1/2}\right\} .
\]
Then 
\[
\mu^{(n)}(E_{\xi}^{c})\le\frac{1}{\epsilon^{2}n}\sum_{x}\mu^{(n)}(x)\left|\left\langle b(x),\xi\right\rangle \right|^{2}\le\delta.
\]
Since for any $g\in\mathcal{S}_{n}(\delta)$, the total variation
distance between $\mu^{(n)}$ and $g\mu^{(n)}$ is less than $\delta$,
we have 
\[
g\mu^{(n)}(E_{\xi})>\mu^{(n)}(E_{\xi})-\delta\ge1-2\delta.
\]
That is $\mu^{(n)}(g^{-1}E_{\xi})>1-2\delta$. Therefore $\mu^{(n)}(g^{-1}E_{\xi})+\mu^{(n)}(E_{\pi_{g}^{\ast}\xi})>1-2\delta+1-\delta>1$,
it follows that 
\[
\left(g^{-1}E_{\xi}\right)\cap E_{\pi_{g}^{\ast}\xi}\neq\emptyset.
\]
Take an element $x_{\xi}\in\left(g^{-1}E_{\xi}\right)\cap E_{\pi_{g}^{\ast}\xi}$.
By the cocycle identity, 
\[
\left\langle b(g),\xi\right\rangle =\left\langle b(gx_{\xi}),\xi\right\rangle -\left\langle \pi_{g}b(x_{\xi}),\xi\right\rangle =\left\langle b(gx_{\xi}),\xi\right\rangle -\left\langle b(x_{\xi}),\pi_{g}^{\ast}\xi\right\rangle .
\]
Since $gx_{\xi}\in E_{\xi}$ and $x_{\xi}\in E_{\pi_{g}^{\ast}\xi}$,
we have 
\[
\left|\left\langle b(g),\xi\right\rangle \right|\le2\epsilon n^{1/2}.
\]
Since this is true for all unit vectors $\xi$, we conclude that for
all $n>N$, any $g\in\mathcal{S}_{n}(\delta)$, $\left\Vert b(g)\right\Vert \le2\epsilon n^{1/2}$.
\end{proof}

\begin{proof}[Proof of Proposition \ref{main_prop}] Denote $\phi:G\to \mathbb{Z}$ the mapping of the exact sequence.
The subgroup $F$ is normal in $N$, thus for $g,g'$ in $G$ with $\phi(g)=\phi(g')$, we have $gFg^{-1}=g'F(g')^{-1}$. Therefore we can write $F^{x}$ for the conjugation $F^{x}=gFg^{-1}$ where $\phi(g)=x \in \mathbb{Z}$. 

As $\mu$ is non-degenerate, there is some integer $R$ such that $F \subset \textrm{supp}(\mu^{(R)})$. We set:
\[
\varepsilon:=|F|\inf_{z \in F} \mu^{(R)}(z)>0.
\]

Given an element $\gamma\in F^{x}$, we consider a coupling of two $\mu^{(R)}$-random
walks starting from $\gamma$ and $e$. We write the two trajectories
as $\gamma Z_n$ and $\tilde{Z}_n$ where $Z_n=X_1\ldots X_n$ and  $\tilde{Z}_n=\tilde{X}_{1}\ldots\tilde{X}_{n}$ are defined as follows, satisfying $\phi(Z_{n})=\phi(\tilde{Z}_{n})$ at all times. As long as case (c) has not happened, we set:
\begin{enumerate}[label=(\alph*)]
\item if $\phi(Z_n)\neq x$, then $\tilde{X}_{n+1}=X_{n+1}$ distributed as $\mu^{(R)}$,
\item if $\phi(Z_n)= x$, then with $1-\varepsilon$ probability $\tilde{X}_{n+1}=X_{n+1}$ distributed as $\frac{1}{1-\varepsilon}\left(\mu^{(R)}-\varepsilon \mathbf{u}_F\right)$, where $\mathbf{u}_F$ is the uniform distribution on the finite group~$F$,
\item if $\phi(Z_n)= x$, with $\varepsilon$ probability $X_{n+1}$ is uniformly distributed in $F$ and $\tilde{X}_{n+1}=\tilde{Z}_{n}^{-1}\gamma Z_nX_{n+1}$, implying $\gamma Z_{n+1}=\tilde{Z}_{n+1}$.
\end{enumerate}

Once case (c) happened, the random walks are coupled and we just set $\tilde{X}_{n+1}=X_{n+1}$ distributed as $\mu^{(R)}$.
By construction, we have 
\begin{eqnarray}\label{position}
\forall n, \quad \tilde{Z}_{n}^{-1}\gamma Z_{n}\in F^{x-\phi(Z_{n})}.
\end{eqnarray}
In particular at the times when the coupling attempt is performed, $\phi(Z_{n})=x$, the difference is in $F$. 

Let $\tau$ be
the coupling time. Every time the projection of the $\mu^{(R)}$-random walk on $\mathbb{Z}$
visits $x$, the chance to succeed is $\varepsilon$.  The distribution of the local time of simple random walk on the integers is known
explicitly. For the $R$-convolution, the local time $L(x,n)=|\{0\le m \le n : \phi(Z_m)=x\}|$, differs by at most a multiplicative constant $C=C(R)$. For
$x\le c\sqrt{n}$, 
\begin{eqnarray*}
\mathbb{P}(\tau >n) &=&\mathbb{P}(L(n,x)=0)+\sum_{k=1}^{n-x}\mathbb{P}(L(n,x)=k)(1-\varepsilon)^{k} \\
&\le& \frac{C}{2^{n}}\sum_{k=0}^{x-1} \left(\begin{array}{c}
n\\
\left\lfloor \frac{n-k}{2}\right\rfloor \end{array}\right) +\sum_{k=1}^{n-x}\frac{1}{2^{n-k}}\left(\begin{array}{c}
n-k\\
\left\lfloor (n+x)/2\right\rfloor
\end{array}\right)
(1-\varepsilon)^{k/C},
\\
& \le & \frac{C}{n^{1/2}}\left(|x|+\frac{1}{\varepsilon}\right),
\end{eqnarray*}
using successively \cite[Theorems 9.1, 9.4 and 2.8]{Revesz2013}. 

It follows that for $|x|\leq r+R$,  for any $\delta>0$, there is an integer $M=M(\delta)>0$ such that 
starting from any  $\gamma$ in such $F^x$, 
$$\mathbb{P}(\tau>Mr^2)<\delta.$$ 
This implies the total variation distance is bounded by
\[
\frac{1}{2}\left\Vert \mu^{(Mr^2)}-g\mu^{(Mr^2)}\right\Vert _{1}\le\mathbb{P}(\tau>Mr^2)\le \delta.
\]
With the notations of Lemma \ref{byOzawa}
\begin{eqnarray}\label{Sdelta}
\bigcup_{|x|\leq r+R} F^x \subset \mathcal{S}_{Mr^2}(\delta).
\end{eqnarray}

Let $\tau_r$ be the exit time of the interval $[-r,r]$. For $|x| \leq cr$, $0<c<1/2$, by standard resistance calculation, we have a similar estimates that for some constant $C'=C'(R)$,
\begin{eqnarray}\label{coupling_by_exit}
\mathbb{P}\left(\tau>\tau_r\right)\le \frac{C'}{r}\left(|x|+\frac{1}{\varepsilon}\right).
\end{eqnarray}

Now let $\gamma$ belong to $F$, so $x=0$. By optional stopping theorem, 
\begin{eqnarray}\label{ost}
\Vert b(\gamma)-b(e)\Vert= \Vert \mathbb{E}[b(\gamma Z_{\tau_r})-b(\tilde{Z}_{\tau_r})]\Vert \le \mathbb{E}\Vert b(\gamma Z_{\tau_r})-b(\tilde{Z}_{\tau_r})\Vert.
\end{eqnarray}
By cocycle identity $b(\gamma Z_{\tau_r})=b(\tilde{Z}_{\tau_r}\tilde{Z}_{\tau_r}^{-1}\gamma Z_{\tau_r})=b(\tilde{Z}_{\tau_r})+\pi_{\tilde{Z}_{\tau_r}}b(\tilde{Z}_{\tau_r}^{-1}\gamma Z_{\tau_r})$. Thus
\[
\Vert b(\gamma)-b(e)\Vert\le \mathbb{E} \Vert b(\tilde{Z}_{\tau_r}^{-1}\gamma Z_{\tau_r})\Vert \leq \mathbb{P}(\tau >\tau_r) \sup\{\Vert b(g)\Vert: g \in F^{-\phi(Z_{\tau_r})}\}.
\]
by (\ref{position}). We can use (\ref{coupling_by_exit}) and (\ref{Sdelta}), as $|\phi(Z_{\tau_r})|\le r+R$, to bound:
\[
\Vert b(\gamma)-b(e)\Vert \le \frac{C'}{\varepsilon r}\sup\{\Vert b(z)\Vert : z \in \mathcal{S}_{Mr^2}(\delta)\}\underset{r \to \infty}{\longrightarrow} 0.
\]
The limit is zero by Lemma \ref{byOzawa}.
\end{proof}

We observe from line (\ref{ost}) that a harmonic function $h:G\to \mathcal{H}$ (not necessarily a cocycle) satisfying $\frac{1}{n}\max\{\Vert h(x)-h(y)\Vert:d(x,y)\le n\}\longrightarrow 0$ must factorize by $F$.

\begin{proof}[Proof of Theorem \ref{main}]
The extension $1\to L \to G \to \mathbb{Z} \to 1$ is necessarily split, as there exists $g_0$ in $G$ with $\phi(g_0)=1$. Thus $G$ admits a symmetric generating set with projection to $\mathbb{Z}$ included in $\{-1,+1\}$. We take $\mu$ supported on it and apply Proposition \ref{main_prop}. 

For every weakly mixing representation $\pi$, any $\mu$-harmonic cocycle $b$ factorizes by $L$, i.e. $b|_L=0$. Moreover, for any $g$ in $G$ and $h$ in $L$, there is $h'$ in $L$ such that $hg=gh'$, so
\[
b(g)=b(g)+\pi_gb(h')=b(gh')=b(hg)=b(h)+\pi_hb(g)=\pi_hb(g).
\]
Therefore restricted to $\overline{\textrm{Im}(b)}$, the representation $\pi$ factorizes by $L$. If $b\neq 0$, then $\pi|_{\overline{\textrm{Im}(b)}}$ admits a trivial subrepresentation as $\mathbb{Z}$ has property $H_{\mathrm{T}}$. This proves $H_{\mathrm{T}}$ for $G$.
\end{proof}

\subsection{A variant on harmonic functions}\label{subsec: variant}

By \cite[Theorem 4.7]{Gournay}, any harmonic cocycle with weakly mixing linear part must vanish on the FC-center. In this subsection we formulate a variant of this result which states that any harmonic function with sub-exponential growth is constant on the subset of torsion elements in the FC-center. 

Recall that the FC-center $Z^{FC}(G)$ is the subgroup of elements with a finite conjugacy class. A function $h:G\to \mathbb{R}$ has {\it subexponential growth} if
\[
e^{-cn}M_h(n) \underset{n \to \infty}{\longrightarrow} 0, \textrm{ for all }c>0,
\]
where $M_h(n)=\max\{|h(x)|:d(e,x)\le n\}$. A function $h:G\to \mathbb{R}$ has {\it sublinear growth} if 
$M_h(n)/n\to 0$ as $n\to\infty$.
\begin{proposition}\label{sub-ex function}
Let $G$ be a finitely generated group with a finitely supported symmetric non-degenerate probability measure $\mu$. Let $h:G\to \mathbb{R}$ be a $\mu$-harmonic function of subexponential growth. Then for any $\gamma$ in $Z^{FC}(G)$ of finite order and for any $x$ in $G$, one has $h(\gamma x)=h(x)$.
\end{proposition}

The proof is an adaptation of the coupling argument used for Proposition~\ref{main_prop}.

\begin{proof}
As the conjugacy class of $\gamma$ is finite, we have $g_1,\dots,g_r$ in $G$ such that
\[
\langle \gamma \rangle^G=\bigcup_{i=1}^r \langle \gamma \rangle^{g_i} \textrm{, and } \langle \gamma \rangle^{g_i} \cap \langle \gamma \rangle^{g_j}=\{e\} \textrm{ whenever } i\neq j,
\]
and this union is a finite set, hence contained in the support of $\mu^{(R)}$ for some finite $R$. Set $\varepsilon:=\textrm{order}(\gamma)\inf\{\mu^{(R)}(x):x\in \langle \gamma \rangle^G\}>0$. 

We couple two $\mu^{(R)}$-random walks $\gamma x Z_n$ and $x\tilde{Z}_n$ as follows. Always sample $X_{n+1}$ according to $\mu^{(R)}$. As long as matching case (b) has not happened $\tilde{Z}_n=Z_n$, so for each $n$ there are integers $i, j$ such that $(x\tilde{Z}_n)^{-1}\gamma xZ_n=(\gamma^j)^{g_{i}}\in \langle \gamma\rangle^{g_{i}}$.
\begin{enumerate}[label=(\alph*)]
\item With probability $1-\varepsilon$, the sample $X_{n+1}$ is not in $\langle \gamma\rangle^{g_{i}}$, then set $\tilde{X}_{n+1}=X_{n+1}$,
\item with probability $\varepsilon$, the sample $X_{n+1}=(\gamma^s)^{g_{i}}$ belongs to $\langle \gamma\rangle^{g_{i}}$, then set $\tilde{X}_{n+1}=(\gamma^{-s-j})^{g_i}$, so $x\tilde{Z}_{n+1}=\gamma x Z_{n+1}$.
\end{enumerate}
Once case (b) occured the increments always agree. The probability not to couple by time $n$ is $(1-\varepsilon)^n$. By the classical martingale argument: 
\[
|h(\gamma x)-h(x)| \leq \mathbb{E}|h(\gamma x Z_n)-h(x \tilde{Z}_n)| \leq (1-\varepsilon)^n M_h(C n) \underset{n \to \infty}{\longrightarrow} 0,
\]
where $C$ is the diameter of the support of $\mu^{(R)}$.
\end{proof}

\subsection{Examples of groups with property $H_{\rm{T}}$}\label{subsec: example}

As mentioned in the Introduction, the groups introduced in \cite{BZ} with prescribed random walk behavior are locally-normally-finite-by-$\mathbb{Z}$, therefore by Theorem \ref{main} they have Shalom's property $H_{\mathrm{T}}$. Here we introduce a modification of the construction, which produces groups which are FC-central extensions of lamplighter groups, moreover elements in the FC-center are torsion. 

These groups are constructed by taking diagonal products of a sequence of marked groups. Let $(G_s)_{s\ge 1}$ be a sequence of marked groups, each $G_s$ marked with a $\ell$-tuple of generators $\mathcal{T}_s=(t_1(s),...,t_{\ell}(s))$. The diagonal product $\Delta$ of $(G_s)_{s\ge 1}$ (also called the universal group of this sequence) is the quotient $\mathbf{F}_{\ell} /\cap_{s}\ker\left(\boldsymbol{\pi}_{s}\right)$, with
the projection map $\boldsymbol{\pi}_s:\mathbf{F}_{\ell} \to G_s$ sending the generators of the free group on $(t_1,...,t_{\ell})$ to the marked generators of $G_s$.

We recall some notations for wreath products. An element in the (restricted) wreath product $H\wr G$ is represented by a pair $(f,x)$ 
 where $f:G\to H$ is a function of finite support and $x\in G$. We refer to $f$ as
the lamp configuration and $x$ as the position of the cursor. The
product rule is 
\[
(f,x)(g,y)=(f\tau_xg,xy),\textrm{ where } \tau_xg(z)=g(zx^{-1}).
\]
The neutral element is denoted as $\left(\boldsymbol{e}, e_G\right)$ where $\mbox{support}(\boldsymbol{e})$
is the empty set. For $x\in G$ and $\gamma\in H$, we denote
by $\gamma\delta_{x}$ the function taking value $\gamma$ at $x$
and $e_{H}$ elsewhere. When $H$ is abelian we write $0$ for its identity element and use additive notation.

Let $A=\{a_{1},\dots,a_{|A|}\}$ and $B=\{b_{1},\dots,b_{|B|}\}$
be two finite groups. Let $\left\{ \Gamma_{s}\right\} $ be a sequence
of groups such that each $\Gamma_{s}$ is marked with a generating
set of the form $A(s)\cup B(s)$ where $A(s)$ and $B(s)$ are finite
subgroups of $\Gamma_{s}$ isomorphic respectively to $A$ and $B$.
We fix the isomorphic identification and write $A(s)=\{a_{1}(s),\dots,a_{|A|}(s)\}$
and similarly for $B(s)$. We impose the assumption 
\[
\Gamma_{s}/[A(s),B(s)]^{\Gamma_{s}}\simeq A(s)\times B(s)\simeq A\times B.
\]
That is, the relative abelianization, which is always a quotient of $A\times B$,
is in fact isomorphic to $A\times B$.

In \cite{BZ}, the factor groups $(G_s)$ are taken to be wreath products $\Gamma_s\wr \mathbb{Z}$ 
with lamp groups $\Gamma_s$ and marked generating tuple carefully chosen to follow prescribed speed function. 
Here instead we take each factor group to be a finite group 
$G_s=\Gamma_s \wr (\mathbb{Z}/m_s\mathbb{Z})$.
The purpose is exactly to produce extensions with finite conjugacy classes. 

The inputs into the construction are two sequences of strictly increasing positive 
integers  $\left(k_{s}\right)_{s\ge 1}$ and  $\left(m_{s}\right)_{s\ge 1}$ such that $m_s\ge 2k_s$ for all $s\ge 1$; 
and a sequence of finite groups $\Gamma_s$ marked with generating set $(A(s),B(s))$ satisfying the assumptions described above.

Take the wreath product $G_s=\Gamma_s \wr \mathbb{Z}/m_s\mathbb{Z}$.
The generating tuples are marked as follows:
\[
\mathcal{T}_{s}=\left(\tau(s),\alpha_{1}(s),\ldots,\alpha_{|A|}(s),\beta_{1}(s),\ldots,\beta_{|B|}(s)\right)
\]
where $\tau(s)=\left(\boldsymbol{e},+1\right)$ and 
\[
\alpha_{i}(s)=\left(a_{i}(s)\delta_{0},0\right),1\leq i\leq|A|,\ \beta_{i}(s)=\left(b_{i}(s)\delta_{k_{s}},0\right),1\leq i\leq|B|.
\]
The group we consider is the diagonal product $\Delta$ of the sequence $((G_s,\mathcal{T}_s))_{s\ge 1}$.

By construction, the group $\Delta$ is an FC-central extension of $(A\times B)\wr \mathbb{Z}$, and the FC-center consists of torsion elements.
\begin{fact}\label{factFC} Let  $\left(k_{s}\right)_{s\ge 1}$, $\left(m_{s}\right)_{s\ge 1}$ and  $(\Gamma_s)_{s\ge 1}$ be as above,
and $\Delta$ be the diagonal product of marked groups $((G_s,\mathcal{T}_s))_{s\ge 1}$. 
Then the FC-center of $\Delta$ is
$$Z^{\mathrm{FC}}(\Delta)=\oplus_{s\ge0}\ker\left(\Gamma_s\to A(s)\times B(s)\right)^{m_s},$$
and
$$\Delta/Z^{\mathrm{FC}}(\Delta)\simeq (A\times B)\wr \mathbb{Z}.$$
\end{fact}

\begin{proof}
In $A\times B \wr \mathbb{Z}$, any non-trivial conjugacy class is infinite. It is therefore sufficient to show that the direct sum is included in the FC-center. Any element $f$ there is a direct sum of functions $f_s$ from $\mathbb{Z}/m_s\mathbb{Z}$ to $\ker\left(\Gamma_s\rightarrow A(s)\times B(s)\right)$ and all these function are trivial for $s >s_0$ depending on $f$. Conjugating by the translation $\tau$ amounts to shift the functions to $f_s(\cdot+1)$. Conjugating by an element $a$ of $A$ (resp. $b$ of $B$) amounts to conjugate the values of the functions at $0$ to $af_s(0)a^{-1}$ (resp. at $k_s$ to $bf_s(k_s)b^{-1}$). As $\Gamma_s$ and $m_s$ are finite, the conjugacy class of $f$ is included in the finite $\oplus_{s\leq s_0}\ker\left(\Gamma_s\rightarrow A(s)\times B(s)\right)^{m_s}$.
\end{proof}

By \cite[Theorem 4.7]{Gournay} or Theorem \ref{main}, $\Delta$ has Shalom's property $H_{\mathrm{T}}$. Recall that a harmonic $1$-cocycle projects to 
Lipschitz harmonic functions on the group. For a group without property $H_{\mathrm{FD}}$, there is an infinite dimensional space of 
non-constant Lipschitz harmonic functions 
coming from $1$-dim projections of harmonic cocycle with weakly mixing representation. 
In general there are Lipschitz harmonic functions that are not related to equivariant harmonic cocycles.
On the diagonal product $\Delta$ constructed above, we have that all its harmonic functions of sub-exponential growth 
factor through the quotient $(A\times B)\wr \mathbb{Z}$. 
More precisely, let $\mu$ be a non-degenerate symmetric probability measure on $\Delta$,
then by Proposition \ref{sub-ex function},
all $\mu$-harmonic functions of sub-exponential growth on $\Delta$ are constant on $Z^{\mathrm{FC}}(\Delta)$. 
In other words, all $\mu$-harmonic functions of sub-exponential growth on $\Delta$ factor through the projection $\Delta\to (A\times B)\wr \mathbb{Z}$.
Since $A\times B$ is finite, the lamplighter group $(A\times B)\wr \mathbb{Z}$ does not have non-constant sublinear $\bar{\mu}$-harmonic functions \cite{BDCKY}, it follows that $\Delta$ doesn't have any non-constant sublinear $\mu$-harmonic functions.

 \begin{proof}[Proof of Corollary \ref{cor-function}]
 
Since speed of random walk on $\Delta$ is bounded from below by speed on each factor $G_s$, given a sub-addtive function $f$ such that $\frac{f(x)}{x}\to 0$ as $x\to \infty$, we can choose $\Gamma_s$ and $k_s$ similar to \cite{BZ} and take $m_s$ to grow sufficiently faster than $k_s$, $m_s\gg k_s$, to guarantee that the speed of random walk on $\Delta$ is faster than $f$. 

\end{proof}

\begin{remark}
By adapting the methods of \cite{BZ}, the speed, return probability and $\ell^p$-isoperimetry of the group $\Delta$ we discuss here can be estimated quite precisely. We don't pursue this direction here. By choosing $m_s$ sufficiently larger than $k_s$, one can guarantee that the random walk behavior (such as speed, entropy, return probability) on the two constructions with factor groups $\Gamma_s\wr \mathbb{Z}$ or $\Gamma_s\wr (\mathbb{Z}/m_s\mathbb{Z})$ with the same $(k_s)$ are comparable. In some sense the random walk parameters don't distinguish these two. However, the diagonal product with factor groups $\Gamma_s\wr \mathbb{Z}$ can admit non-constant sublinear harmonic functions. This aspect will be addressed elsewhere.
\end{remark}

This construction also permits to disprove one direction of a conjecture of Shalom that among solvable groups, property  $H_{\mathrm{FD}}$ would be equivalent to infinite Hirsch length.

\begin{proposition}\label{3solv}
There exists a $3$-solvable group with Shalom's property $H_{\mathrm{T}}$ and infinite Hirsch length.
\end{proposition}

\begin{proof}
Let $(k_s)_{s\geq 1}$ and $(m_s)_{s\geq 1}$ be as above. Take $A=B=\mathbb{Z}/2\mathbb{Z}$ and $\Gamma_s=D_{\infty}$ to be an infinite dihedral group for all $s$. The associated diagonal product $\Delta$ is $3$-solvable and contains $\oplus_{s\leq s_0}\ker\left(\Gamma_s\rightarrow A(s)\times B(s)\right)^{m_s}=\oplus_s\oplus_{\mathbb{Z}/m_s\mathbb{Z}} \mathbb{Z}$ so has infinite Hirsch length. Moreover, this subgroup coincides with the FC-center.

Indeed, $\ker\left(D_\infty\rightarrow A\times B\right)$ is a cyclic group generated by the commutator $abab$ of two involutions. Non-trivial elements there have a conjugacy class of size two in $D_\infty$. The proof of Fact \ref{factFC} applies. Again, the group $\Delta$ is an extension of its FC-center by $A\times B \wr \mathbb{Z}$, hence has Shalom's property $H_{\mathrm{T}}$ by \cite{Gournay}.
\end{proof}

\section{Examples of groups without property $H_{\mathrm{FD}}$}

In this section we construct explicit harmonic cocycles on some families of groups. Many of them are obtained as virtual coboundaries. 
Recall that given a unitary representation $\pi: G\to \mathcal{H}$, a cocycle $b:G\to \mathcal{H}$ is called a virtual coboundary if 
$b(g)=\pi(g)x-x$ for some $x\in W\setminus \mathcal{H}$, where $W$ is a vector space where the unitary representation $\pi$ extends to a linear action on $W$.
Finding virtual coboundaries is a useful tool to exhibit cocycles with certain properties, see for example \cite{FV1}, \cite{Gournay}.  

\subsection{Wreath products $(\mathbb{Z}/2\mathbb{Z})\wr \mathbb{Z}^d$ with $d\ge 3$ do not have~$H_{\mathrm{FD}}$}\label{wrZd}

We show a more general result. 
\begin{proposition}\label{wreath3}
Suppose $H$ is a finitely generated group on which simple random walk is transient. The wreath product $(\mathbb{Z}/2\mathbb{Z})\wr H$ does not have property $H_{\mathrm{FD}}$.
\end{proposition}

\begin{proof}
Denote elements of $G=(\mathbb{Z}/2\mathbb{Z})\wr H$ by $g=(f,h)$, where $f:H\to \mathbb{Z}/2\mathbb{Z}$ is the lamp configuration and $h\in H$. Let $W$ be the vector space of real valued functions on $H$. Take the linear representation $\pi$ of $G$ on $W$ defined by
$$\pi((f,h))\psi(x)=(-1)^{f(x)}\psi(xh).$$
The restriction of $\pi$ to the Hilbert space $\ell^2(H)$ is a unitary representation. It is clear that $\pi$ doesn't have any finite dimensional sub-representation because the subgroup $H$ acts by the right regular representation on $\ell^2(H)$. 

Let $\mu$ be a symmetric probability measure on $H$ with finite generating support. By assumption that the $\mu$-random walk is transient we have that there is a unit flow of finite energy from identity $e_H$ to infinity, see \cite[Theorem 2.11]{book}. Take the corresponding voltage function $v: H\to \mathbb{R}$, then it satisfies that $v(e_H)=1$, 
$$(I-P_{\mu}) v=\mathbf{1}_{e_H},$$
where $P_{\mu}v(x)=\sum_{y\in H}v(xy)\mu(y)$.
In particular $v$ is $\mu$-harmonic on $H$ except at identity $e_H$. 

Add a constant to $v$ so that $v(e_H)=a$, where $a\in\mathbb{R}$ is a number to be chosen later. 
Take a virtual coboundary $b:G\to \ell^2(H)$ defined as 
$$b(g)=v-\pi(g)v.$$
We check that $b(g)$ is indeed in $\ell^2(H)$: for a generator $s\in \mbox{supp}\mu$, 
$$\left\Vert b((0,s)) \right\Vert_{\ell^2(H)}^2=\sum_{x\in H} (v(x)-v(xs))^2 \le \frac{1}{\mu(s)}\sum_{x,y\in H} (v(x)-v(xy))^2\mu(y)<\infty.$$
In the last inequality we have used that the flow of $v$ has finite energy by the transience assumption. 
And for $g=(f,e_H)$, since $f$ is of finite support, 
$$\left\Vert b((f,e_H)) \right\Vert_{\ell^2(H)}^2=\sum_{x\in \rm{supp}f}4v(x)^2<\infty.$$
Since $b(g)\in \ell^2(H)$ for generators of $G$, it follows that $b(g)$ is in $\ell^2(H)$ for any $g\in G$.

Take $\eta=\frac{1}{2}\mu+\frac{1}{2}\mathbf{1}_{\delta_e^1}$, where $\delta_e^1$ denotes the element $(f,0)$ where $f=1$ at $e_H$ and $0$ otherwise. It is clear that $\eta$ is a symmetric probability measure on $G$ with generating support. We now choose the constant $a$ to make $b$ an $\eta$-harmonic cocycle. To this end it suffices to have
$$\sum_{g\in G}b(g)\eta(g)=0.$$
Note that 
$$b(\delta_e^1)=2a\mathbf{1}_{e_H},\ \sum_{h\in H}b((0,h))\mu(h)=(I-P_{\mu})v=\mathbf{1}_{e_H}.$$
Thus by choosing $a=-1/2$, we have that $b:G\to \ell^2(H)$ is a $\eta$-harmonic cocycle with weakly mixing representation $\pi$. It follows that $G$ doesn't have property $H_{\rm{FD}}$.

\end{proof}

Note that by Varopoulos \cite{var}, the only finitely generated groups that carry recurrent simple random walks are the finite extensions of $\{e\}, \mathbb{Z}$ or $\mathbb{Z}^2$. In the Appendix A we show that the wreath product $H\wr G$ of two finitely generated infinite groups where $H$ is amenable doesn't have property $H_{\mathrm{FD}}$. Combined with the previous proposition and the fact that $F\wr \mathbb{Z}$ has property $H_{\mathrm{FD}}$ for $F$ finite by \cite{Shalom}, we have that for wreath product $H\wr G$ of finitely generated groups, where $H$ is amenable and $G$ is infinite, the only case where it is not known whether $H\wr G$ has property $H_{\mathrm{FD}}$ is when $H$ is finite and $G$ is a finite extension of $\mathbb{Z}^2$.

\subsection{Locally-finite-by-$\mathbb{Z}$ groups without property $H_{\mathrm{FD}}$}\label{sec:noHFD}

In this subsection we give examples of locally-finite-by-$\mathbb{Z}$ groups which do not have property $H_{\mathrm{FD}}$. This answers negatively a question of Shalom \cite[Section 6.6]{Shalom}.

\begin{proposition}\label{noH_FD}
Let $G$ be a finitely generated locally-finite-by-$\mathbb{Z}$ group and $\mathcal{S}$ be a Schreier graph of $G$. Assume $\pi:G\rightarrow \mathcal{U}\left(\ell^2(\mathcal{S})\right)$ is a unitary representation fixing no non-zero vector. If there exists an associated non-zero harmonic cocycle $b:G \rightarrow \ell^2(\mathcal{S})$, then $G$ does not have Shalom's property $H_{\mathrm{FD}}$.
\end{proposition}

This proposition is immediate once we recall.

\begin{fact}[\cite{Shalom}]\label{423}
A finitely generated locally finite by $\mathbb{Z}$ group $G$ with property $H_{\mathrm{FD}}$ has property $H_{\mathrm{T}}$.
\end{fact}

\begin{proof}[Proof of Fact \ref{423}]
By \cite[Proposition 4.2.3]{Shalom}, the group $G$ has property $H_F$ if and only if no finite index subgroup $G_0$ admits a homomorphism to $S^1 \ltimes \mathbb{C}$ with dense image. However up to finite index the image must be torsion-free, hence cyclic so cannot have dense image. So $G$ has property $H_F$. Property $H_{\mathrm{T}}$ follows from \cite[Proposition 4.2.4]{Shalom} as the first Betti number and the first virtual Betti number of $G$ agree.
\end{proof}

\begin{proof}[Proof of Proposition \ref{noH_FD}]
Existence of a harmonic cocycle implies $\bar{H}^1(G,\pi) \neq 0$. If $G$ had property $H_{\mathrm{FD}}$, Fact \ref{423} would upgrade it to property $H_{\mathrm{T}}$, providing a non-zero fixed vector for $\pi$.
\end{proof}

\begin{proposition}\label{SymZ}
The group $\rm{Sym}(\mathbb{Z}) \rtimes \mathbb{Z}$ does not have Shalom's property $H_{\mathrm{FD}}$.
\end{proposition}

\begin{proof}
This groups acts by permutations on $\mathbb{Z}$. It is generated by the shift $x\mapsto x+1$ and the transposition permuting $0$ and $1$. The associated Schreier graph is easily pictured. let $\mu$ be the probability giving mass $\frac{1}{4}$ to the shift and its inverse  and mass $\frac{1}{2}$ to the transposition. It is immediate to check that the function
\[
h(x)=\left\{\begin{array}{ll} x & \forall x \leq 0, \\x-\frac{2}{3} & \forall x \geq 1,
\end{array}\right.
\]
is $\mu$-harmonic. It follows that $b(g)(x)=h(x.g)-h(x)-T_g$, where $T_g$ is the translation part of $g$, defines a non-zero $\mu$-harmonic cocycle with respect to the right regular representation on $\ell^2(\mathcal{S})$. The latter has no non-zero fixed vector as the action on this Schreier graph is transitive. Proposition \ref{noH_FD} applies.
\end{proof}

Let $F$ be a finite group acting faithfully transitively on an alphabet $\{0\dots q-1\}$ with $q$ letters. Given a $q+1$-regular tree $\mathbb{T}$, an end of this tree, a bi-infinite ray towards this end and a vertex on this ray, the discrete affine group $\rm{DA}_F(\mathbb{T})$ is the group of automorphisms of the tree generated by an automorphism that shifts the ray together with the $F$-permutations of the $q$ subtrees (not containing the end) obtained by removing the chosen vertex. This group was introduced by Ryokichi Tanaka and the authors in \cite{BTZ}, to which we refer for details. By \cite[Proposition 3.1(3)]{BTZ}, the discrete affine group is locally finite by $\mathbb{Z}$.

\begin{proposition}\label{DA}
The discrete affine group $\rm{DA}_F(\mathbb{T})$ of a $q+1$-regular tree does not have Shalom's property $H_{\mathrm{FD}}$.
\end{proposition}

\begin{proof}
The discrete affine group is generated by a shift of infinite order together with a copy of $F$. The Schreier graph of its action on $\mathbb{T}$ can be pictured as follows. We first describe a core that will support the gradient of our harmonic function. The vertex set consists in all finite words in the alphabet $\{0 \dots q-1\}$. For any such word $v$ (including the empty word) there is a vertical edge between $v$ and $0v$ and a horizontal edge between $iv$ and $jv$ for any $i,j$ in the alphabet. Moreover above each vertex in this core not starting by $0$ we attach a vertical infinite one-sided ray, see Figure \ref{fig1} for the case of a ternary tree. The action of $F$ is by permutation of $\{0v\dots (q-1)v\}$ in the obvious way and $t$ translate along vertical lines.

Let $\mu$ give mass $\frac{1}{4}$ to the shift and its inverse and subprobability $\frac{1}{2}\mathbf{u}_F$ to the copy of $F$. Then vertical edges have mass $\frac{1}{4}$ and horizontal edges have mass $\frac{1}{2q}$ each.

Given $a>0$ and an initial value $h(\emptyset)$, we define a function $h:\mathcal{S}\rightarrow \mathbb{R}$ by the following gradient rules on the core. 
\[
\forall i,j \in\{1,\dots,q-1\}, \forall v, \left\{ 
\begin{array}{l}
h(iv)-h(jv)=0, \\
h(iv)-h(0v)=\frac{a}{2q^{|v|}}, \\
h(0v)-h(v)=\frac{a}{q^{|v|}}.
\end{array}
\right. 
\]
where $|v|$ is the length of $v$. We extend it as constant on each infinite one-sided ray. By construction, $h$ is $\mu$-harmonic and by elementary computation, it has finite energy:
\[
\frac{1}{2}\sum_{x\in \mathcal{S}, g \in \rm{DA}} \left(h(x.g)-h(x)\right)^2\mu(g)=\sum_{n=0}^\infty \left(\frac{a}{q^n} \right)^2\frac{q^n}{4}+\left(\frac{a}{2q^{n+1}}\right)^2\frac{q^n}{2q}\frac{q(q-1)}{2}<\infty.
\]
Now set $\forall x \in \mathcal{S}$, $b(g)(x)=h(x.g)-h(x)$. Then $b$ is a non-zero $\mu$-harmonic cocycle with respect to the right regular representation of $\rm{DA}_F(\mathbb{T})$ acting on $\ell^2(\mathcal{S})$. Proposition \ref{noH_FD} applies.
\end{proof}

\begin{figure}
	\begin{center}
\begin{tikzpicture}
\draw[dotted] (-0.5,-0.75)--(7.5,-0.75);
\draw[thick] (0,0)--(1,0);
\draw (0,0) node[below]{$000$};
\draw (1,0) node[below]{$100$};
\draw[thick,dashed,blue] (1,0)--(1.25,0.5);

\draw[thick] (2,0)--(3,0);
\draw (2,0) node[below]{$010$};
\draw (3,0) node[below]{$110$};
\draw[thick,dashed,blue] (3,0)--(3.25,0.5);

\draw[thick] (4,0)--(5,0);
\draw (4,0) node[below]{$001$};
\draw (5,0) node[below]{$101$};
\draw[thick,dashed,blue] (5,0)--(5.25,0.5);

\draw[thick] (6,0)--(7,0);
\draw (6,0) node[below]{$011$};
\draw (7,0) node[below]{$111$};
\draw[thick,dashed,blue] (7,0)--(7.25,0.5);

\draw[thick] (0,0)--(0.5,1);
\draw (0.25,0.5) node[left,red]{$\frac{a}{q^2}$};
\draw (0.5,1) node[left]{$00$};
\draw[thick] (2,0)--(2.5,1);
\draw (2.25,0.5) node[left,red]{$\frac{a}{q^2}$};
\draw (2.5,1) node[right]{$10$};
\draw[thick] (0.5,1)--(2.5,1);
\draw (1.5,1) node[above,red]{$\frac{a}{2q^2}$};
\draw[thick,dashed,blue] (2.5,1)--(2.875,1.75);
\draw[thick] (4,0)--(4.5,1);
\draw (4.25,0.5) node[left,red]{$\frac{a}{q^2}$};
\draw (4.5,1) node[left]{$01$};
\draw[thick] (6,0)--(6.5,1);
\draw (6.25,0.5) node[left,red]{$\frac{a}{q^2}$};
\draw (6.5,1) node[right]{$11$};
\draw[thick] (4.5,1)--(6.5,1);
\draw (5.5,1) node[above,red]{$\frac{a}{2q^2}$};
\draw[thick,dashed,blue] (6.5,1)--(6.875,1.75);
\draw[thick] (0.5,1)--(1.5,2.5);
\draw (1,1.75) node[left,red]{$\frac{a}{q}$};
\draw (1.5,2.5) node[left]{$0$};
\draw[thick] (5.5,2.5)--(4.5,1);
\draw (5,1.75) node[left,red]{$\frac{a}{q}$};
\draw (5.5,2.5) node[right]{$1$};
\draw[thick] (1.5,2.5)--(5.5,2.5);
\draw (3.5,2.5) node[above,red]{$\frac{a}{2q}$};
\draw[thick,dashed,blue] (5.5,2.5)--(6,3.25);
\draw[thick] (1.5,2.5)--(3.5,5);
\draw (3.5,5) node[right]{$\emptyset$};
\draw (2.5,3.75) node[left,red]{$a$};
\draw[thick,dashed,blue] (4,5.625)--(3.5,5);

\end{tikzpicture}
	\end{center}
	\caption{The Schreier graph $\mathcal{S}$ of $\rm{DA}(\mathbb{T})$ acting on the ternary tree. The core (in this case a Fibonacci tree) is drawn in black and infinite one-sided rays are sketched in blue dashed. Edges are labelled in red with the gradient of the harmonic function $h$.}
	\label{fig1}
\end{figure}
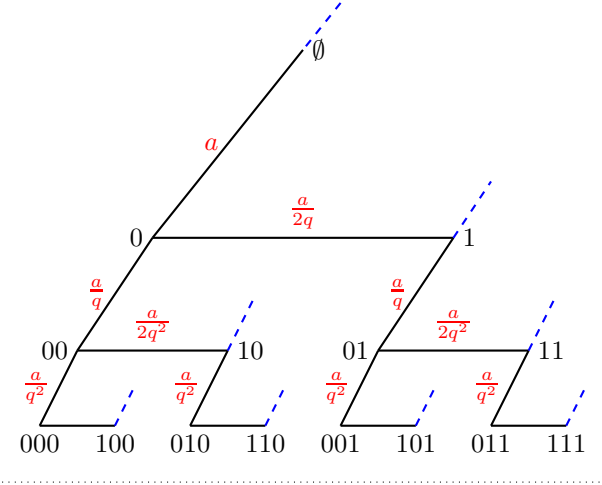

\begin{remark}
In contrast with Proposition \ref{DA}, we observe that the whole affine group of a regular tree (i.e. the group of automorphisms fixing an end--see  \cite{CKW} or \cite{BTZ}) has Shalom's Property $H_{\mathrm{FD}}$ as a topological group. This is proved along the same lines as \cite[Corollary 5.2.7]{Shalom}.
\end{remark}

\begin{remark}
These examples also answer the 4th item in \cite[Question 6.6]{Gournay} negatively. The question asked if $\pi$ is weakly mixing and there is an infinite $N\triangleleft G$ with $G/N$ cyclic and $\overline{H}^1(N,\pi_{|N})=0$, does $\overline{H}^1(G,\pi)=0$? Note that when $N$ is locally finite, it is a direct limit of finite groups, therefore $\overline{H}^1(N,\sigma)=0$ for any unitary representation $\sigma$ of $N$, see for example \cite[Lemma 5]{FV}. Therefore examples of locally-finite-by-$\mathbb{Z}$ groups without property $H_{\mathrm{FD}}$ answer this question negatively.
\end{remark}

We now apply the result on the discrete affine groups to a family of
locally-nilpotent-by-$\mathbb{Z}$ groups introduced by Gromov \cite[Section8.2]{Gromov}.
These groups give examples of elementary amenable groups with arbitrary fast growing F{\o}lner functions.

We first recall the construction in the special case of $\mathbb{Z}$
acting on itself.
Fix $q$ an integer or infinity. Let $\mathbf{F}_{\mathbb{Z}}(q)$ be the free product
of $\mathbb{Z}$ copies of $\mathbb{Z}/q\mathbb{Z}$, 
$\mathbf{F}_{\mathbb{Z}}(q)=\ast_{i\in\mathbb{Z}}\left\langle a_{i}\right\rangle $,
where $\left\langle a_{z}\right\rangle =\mathbb{Z}/q\mathbb{Z}$.
Given a subset $Y\subset\mathbb{Z}$, let $[Y]_{k}$ be the set of
commutators of $k$ letters $\left[\left[a_{i_{1}},a_{i_{2}}\right]\ldots a_{i_{k}}\right]$
where $i_{j}\in Y$ for $1\le j\le k$. Let $\mathbf{k}:2^{\mathbb{Z}}\to\mathbb{N}$
be a function on finite subsets of $X$ satisfying:
\begin{enumerate}
\item $\mathbf{k}(Y)=\mathbf{k}(Y+z)$ for any $Y\subset\mathbb{Z}$ and $z\in\mathbb{Z}$,
that is $\mathbf{k}$ is invariant under translation of $\mathbb{Z}$.
\item $\mathbf{k}(Y_{1})\le\mathbf{k}(Y_{2})$ if $Y_{1}\subseteq Y_{2}$. 
\end{enumerate}
Let $N(\mathbf{k})$ be the quotient of $\mathbf{F}_{\mathbb{Z}}(q)$
with the commutator relations $\left\{ [Y]_{\mathbf{k}(Y)}:\ Y\in2^{X}\right\} .$
Since the function $\mathbf{k}$ is $\mathbb{Z}$-invariant, we can
take the semi-direct product $N(\mathbf{k})\rtimes\mathbb{Z}$ where
$\mathbb{Z}$ acts by translating indices. When $\mathbf{k}(Y)$ is
finite for all finite subset $Y\subset\mathbb{Z}$, the group $N(\mathbf{k})$
is locally nilpotent and locally finite when $q$ is also finite.

As in \cite{Gromov}, we consider the following kind
of function $\mathbf{k}$: 
\[
\mathbf{k}(Y)=\kappa(\mbox{Diam}(Y)),
\]
where $\mbox{Diam}(Y)=\max\{|i-j|:\ i,j\in Y\}$ and $\kappa:\mathbb{N}\cup\{0\}\to\mathbb{N}$
is a non-decreasing function with $\kappa(0)=2$. Denote by $G(q,\kappa)$
the group $N(\mathbf{k})\rtimes\mathbb{Z}$ described above.

\begin{proposition}\label{prop_Gromov}
Let $q$ be a prime. For the group $G=G(q,\kappa)$ described above, we have
\begin{description}
\item [{(i)}] If $\kappa$ is bounded, then $G$ has Property
$H_{\mathrm{T}}$.
\item [{(ii)}] If $\kappa(n)>q^{n+1}$ for all $n$, then $G$ doesn't
have Property $H_{\mathrm{FD}}$.
\end{description}
\end{proposition}

\begin{proof}

We apply Proposition \ref{main_prop} inductively to prove
(i). Let $m=\sup_{n}\kappa(n)$, by assumption of (i) it is finite.
When $m=2$, the group $G(q,\kappa)$ is the lamplighter $(\mathbb{Z}/q\mathbb{Z})\wr\mathbb{Z}$,
it has property $H_{\mathrm{T}}$ by \cite{Shalom}.
Suppose the claim is true for $m$ and $\sup_{n}\kappa(n)=m+1$. Let
$b$ be a $\mu$-harmonic $1$-cocycle on $G$ with weakly
mixing representation. The center $Z(N(\mathbf{k}))$ of $N(\mathbf{k})$
is a locally normally finite subgroup of $N(\mathbf{k})$, therefore
by Proposition \ref{main_prop}, $b=0$ on $Z(N(\mathbf{k}))$.
In other words $b$ factors through to quotient group $G(q,\kappa')=(N(\mathbf{k})/Z(N(\mathbf{k})))\rtimes\mathbb{Z}$,
where $\kappa'=\max\{\kappa,m\}$. Then by the induction hypothesis,
$b=0$ on $G$. 

For (ii), we show that when $\kappa(n)>q^{n+1}$ for all $n$, the
discrete affine group $\textrm{DA}_{\mathbb{Z}/q\mathbb{Z}}(\mathbb{T}_{q})$ is a quotient of $G(q,\kappa)$.
Since $q$ is a prime, the subgroup $\rm{Hor}(\mathbb{T})$ is a
locally finite $q$-group. The subgroup generated by $\left\{ t^{-i}a_{0}t^{i}:\ 0\le i\le\ell\right\} $
is a finite $q$ group of cardinality $q^{q^{\ell+1}}$. Note that
its nilpotency class is bounded by $q^{\ell+1}$. It follows
that $\textrm{DA}_{\mathbb{Z}/q\mathbb{Z}}(\mathbb{T}_{q})$ is a quotient of $G(q,\kappa)$ if $\kappa(n)\ge q^{n+1}+1$
for all $n$. The statement of (ii) follows since by Proposition \ref{DA},
$\textrm{DA}_{\mathbb{Z}/q\mathbb{Z}}(\mathbb{T}_{q})$ does not have property $H_{\mathrm{FD}}$. 

\end{proof}

\section{An embedding result}\label{NNemb}

The goal of this section is to show Proposition \ref{NN}.
Let $C$ be the countable universal locally finite group of P. Hall \cite{Hall}.
Recall that $C$ can be constructed as the direct union of $\left(L_{n}\right)$,
where $L_{1}=\mathbb{Z}/3\mathbb{Z}$, $L_{n+1}$ is the full symmetric
group on $\left|L_{n}\right|$ elements, and $L_{n}$ is embedded
in $L_{n+1}$ via the regular representation. Moreover $C$ is a simple group, and there exists a set of generators
$\left\{ c_{1},c_{2},\ldots\right\} $ of $C$ such that $c_{s}^{2}=1$
for all $s$. 
Since any countable locally finite group $H$ embeds into $C$, Proposition \ref{NN} follows from:

\begin{proposition}\label{NNC}
There exists a
two generated locally-finite-by-$\mathbb{Z}$ group $G$ with Shalom's
property $H_{\mathrm{FD}}$ and containing $C$ as an embedded subgroup.
\end{proposition}

Before proceeding to the proof, first observe that up to enlarging the generating sequence $(c_s)$ of involutions, we can assume that for infinitely many $s$, the subgroup $F_s=\langle c_1,\dots,c_s\rangle$ is simple.

Indeed given finitely many involutions in $C$, the group they generate is finite so there exists an abstract finite simple group generated by involutions that contains it. By universality, a copy $A_s$ of this abstract group lies in $C$. By \cite[Theorem 1(ii)]{Hall}, there is some $g$ in $C$ such that $F_s^g$ is contained in $A_s$. Then $A_s^{g^{-1}}$ contains $F_s$ and is generated by old and new involutions.

Recall that given two groups $H$ and $\Gamma$, their restricted wreath product is the group $H \wr \Gamma=\left( \oplus_{\Gamma} H \right) \rtimes \Gamma$, where $\Gamma$ acts on the direct sum by permutation of the factors and their unrestricted wreath product is the group $H \wr\wr \Gamma=H^{\Gamma} \rtimes \Gamma$ with the same action. Their elements are pairs $(f,g)$ where $f$ is a function $\Gamma\rightarrow H$, finitely supported for the restricted $H \wr \Gamma$, and $g$ an element of the base group $\Gamma$. 

Let $(k_s)_{s \geq 1}$ be an increasing sequence of integers. Define a function $f : \mathbb{Z} \to C$ by
\[
f(x):=\left\{\begin{array}{ll} c_s & \textrm{ if } x=k_s, \\
e & \textrm{ otherwise.}
\end{array} \right.
\]

\begin{fact}\label{copy_C}
Assume $k_{s+1}>2k_s$ for all $s$. Then the group $\langle f,t\rangle<C \wr\wr \mathbb{Z}$ contains an embedded copy of $C$ and is locally-finite-by-$\mathbb{Z}$.
\end{fact}

\begin{proof}
By simplicity, any $c$ in $C$ is a finite product of commutators $c=\prod [c_{s_j},c_{r_j}]$. Then $\prod[f^{k_{s_j}},f^{k_{r_j}}]$ is a function taking value $c$ in $0$ and trivial elsewhere by strict doubling of $(k_s)$. This gives a copy of $C$.

Let the element $(\varphi,0)$ belong to the kernel of the map $\langle f,t \rangle \to \mathbb{Z}$ be of length less than $R$ with respect to the ambiant word metric. Then for all $x \in [-2R,2R]$, $\varphi(x)$ is a word of length at most $R$ in the finite group $C_R=\langle c_1,\dots, c_{\log_2 R}\rangle$. And whenever $k_s-k_{s-1} \ge R$ any $x \in [k_s-R,k_s+R]$ satisfies that either $\varphi(k_s+x)=c_s$ for all $s$ or $\varphi(k_s+x)=e_C$ for all $s$. Moreover the function $\varphi$ is trivial outside of these intervalles. It follows that the group generated by these elements embeds into the finite group $C_r^{[-2R,2R]}\times \mathbb{Z}/2\mathbb{Z}^{[-R,R]}$.
\end{proof}

To prove Proposition \ref{NNC}, we will show that $\langle f,t\rangle$ has property $H_{\mathrm{FD}}$ when the sequence $(k_s)$ grows sufficiently fast. The key point is the following:

\begin{lemma}\label{random_wreath}
For any $\delta>0$, there exists a constant $c>0$ such that the following holds. Let $F$ be a finite group and $\mu$ a probability measure on $F \wr \mathbb{Z}$ with finite generating support such that $\mu(t)=\mu(t^{-1})=\frac{1}{4}$ and $\mu(\oplus_{\mathbb{Z}} F)=\frac{1}{2}$. Then there exists an integer $M$ depending on $F$ and $\mu$ such that 
\[
\frac{1}{2}\Vert \mu^{(n)}-\mu^{(n+\delta n)} \Vert_1 \leq 1-c, \textrm{ for all } n \geq M.
\]
\end{lemma}

\begin{proof}[Proof of Lemma \ref{random_wreath}]
From the distribution of  local time \cite[Theorem 9.4]{Revesz2013}, we obtain for any $\delta>0$ there exists constants $c_1,c_2,c_3>0$ such that
\begin{eqnarray}\label{Trotter}
\mathbb{P}(\forall x \in [-c_1\sqrt{n},c_1 \sqrt{n}], L(x,n)\geq c_2 n^{1/2-\delta})\geq c_3>0.
\end{eqnarray}

There also exists an integer $R$ such that $\mu^{(R)}$ contains the copy of $F$ at $0$ in its support. 
Following the proof of Proposition \ref{main_prop}, we set $\varepsilon=|F|\inf_{z \in F} \mu^{(R)}(z)>0$.

Under the condition in (\ref{Trotter}), for $|x|\le c_1\sqrt{n}$, the law of the lamp at $x$  at time $n$ of the random walk has the form
\[
(1-(1-\varepsilon)^{c_2\frac{n^{1/2-\delta}}{R}})\mathbf{u}_{F}+(1-\varepsilon)^{c_2\frac{n^{1/2-\delta}}{R}}\nu,
\]
for an auxiliary probability $\nu$. Thus for $n$ large enough, all the lamps in this interval are $F$-uniformly randomized with probability at least $\frac{c_3}{2}$. We can even assume that the cursor $\phi(Z_n)$ moreover belongs to $[-\frac{c_1}{2}\sqrt{n},\frac{c_1}{2}\sqrt{n}]$ with some uniform positive probability $c'_3>0$.

This gives a subprobability $\zeta_n \le \mu^{(n)}$ of mass $c_3'>0$ under which all lamps in $[-c_1\sqrt{n},c_1 \sqrt{n}]$ are $F$-uniformly randomized and the cursor $\phi(Z_n)$ lies in $[-\frac{c_1}{2}\sqrt{n},\frac{c_1}{2}\sqrt{n}]$.

For $\delta>0$, there is a probability $c_4>0$ that the cursor remains in $[-c_1\sqrt{n}+R',c_1 \sqrt{n}-R']$ up to time $n+\delta n$ and finishes again in $[-\frac{c_1}{2}\sqrt{n},\frac{c_1}{2}\sqrt{n}]$. Therefore, $\frac{c_4}{c'_3}\zeta_n$ is a subprobability of uniform mass $c=c_4>0$ of both $\mu^{(n)}$ and $\mu^{(n+\delta n)}$.
Here $R'$ is such that $f(x)=e$ for all $|x|>R'$ and $(f,0)$ in the support of~$\mu$.
\end{proof}

\begin{proof}[Proof of Proposition \ref{NNC}]
Let $s$ be such that $\langle c_1,\dots,c_s\rangle$ is simple.
Let $F=\langle c_1,\dots c_s\rangle \times \mathbb{Z}/2\mathbb{Z}$ and $f:\mathbb{Z}\to F$ be given by:
\[
f(x):=\left\{\begin{array}{ll} (c_s,0) & \textrm{ if } x=k_s, s\ge 2, \\
(c_1,1) & \textrm{ if } x=k_1, \\
(e,0) & \textrm{ otherwise.}
\end{array} \right.
\]
The argument of Fact \ref{copy_C} shows that $t$ and $f$ generate $F\wr \mathbb{Z}$, so Lemma \ref{random_wreath} applies.
We observe that the groups generated by $f$ and $t$ in $F\wr \mathbb{Z}$ and $C \wr \wr \mathbb{Z}$ coincide on any ball of radius $<k_{s+1}-k_s$. The factor $\mathbb{Z}/2\mathbb{Z}$ takes into account the generators $c_{s+1},c_{s+2}\dots$ not interplaying together in this ball. Therefore, if we let $(k_s)$ grow fast enough, we can use Lemma \ref{random_wreath} locally and obtain a sequence $n_s\to \infty$ with 
\[
\frac{1}{2}\left\Vert \mu^{(n_s)}-\mu^{(n_s+\delta n_s)} \right\Vert_1 \leq 1-c, \textrm{ for all } s.
\]
Proposition \ref{NNC} follows from \cite[Corollary 2.5]{EO}.
\end{proof}

We remark that by the proof, the group $G$ can be chosen with speed arbitrarily close to diffusive, namely below $\sqrt{n}f(n)$ for any $f(n)\to \infty$.

\bigskip
\textbf{Acknowledgement.}  We thank Anna Erschler for questions about Shalom's property $H_\mathrm{FD}$ which motivate this work, and for pointing out to us the examples of locally-nilpotent-by-$\mathbb{Z}$ groups in \cite{Gromov}. We thank Gidi Amir for many discussions on minimal growth of harmonic functions on groups. We thank Peter Kropholler for useful comments. 

J.B.\ is partially supported by project AGIRA ANR-16-CE40-0022.

\appendix

\section{Wreath products of infinite groups do not have property $H_{\mathrm{FD}}$}\label{appendix}

In \cite[Section 6.6]{Shalom}, Shalom conjectured that the wreath product of two infinite finitely generated amenable groups never has property $H_{\mathrm{FD}}$. The object of this appendix is to prove this result. We will follow largely Shalom's original paper, where the conjecture was already proven when $H$ has infinite abelianization \cite[Section 5.4]{Shalom}.
As an abuse of notation, we will write $H^G$ instead of $\oplus_{G} H$ in this appendix.

\begin{proposition}
If $G,H$ are finitely generated infinite groups and $H$ is amenable, then $H \wr G$ does not have Shalom's property $H_{\mathrm{FD}}$.
\end{proposition}

\begin{proof}
Let $\mu$ and $\nu$ be  non-degenerate finitely supported symmetric probability measures on $H$ and $G$ respectively. They induce a measure $\theta =\frac{1}{2}(\mu+\nu)$ on $H \wr G$ where we identify $H$ with its copy over the neutral element of $G$ in $H^G \subset H \wr G$.

When $H$ admits an infinite abelianization, the proposition was proved by Shalom in \cite[Theorem 5.4.1]{Shalom}. We give the proof for completeness. Denote $\phi: H \rightarrow \mathbb{Z}$ a surjective homomorphism and let $\pi:G \rightarrow \mathcal{U}(\ell^2(G))$ be the right regular representation of $G$. It is weakly-mixing as $G$ is infinite.  Then $b(f,g)(x)=\phi \circ f(x)$ defines a non-zero $\theta$-harmonic cocycle with respect to the (weakly-mixing) representation of $H\wr G$ factoring through $\pi$.

We are left with the case when $H$ has finite abelianization. We consider a unitary representation $\pi_0:H \rightarrow \mathcal{U}(\mathcal{H})$ with an associated $\mu$-harmonic cocycle $b_0$, which exist by amenability of $H$, see \cite[Appendix]{Kleiner}. Define $b : H\wr G \rightarrow \ell^2(G,\mathcal{H})$ by $b(f,g)(x)=b_0(f(x))$. This is a non-zero $\theta$-harmonic cocycle for the unitary representation $\pi(f,g) \varphi(x)=\pi_0(f(x))\varphi(x.g)$. In particular, $\overline{H^1}(H\wr G,\pi) \neq 0$.

Assume by contradiction that $H\wr G$ has property $H_{\mathrm{FD}}$. Then we obtain a direct sum $\ell^2(G,\mathcal{H})=V'\oplus \bigoplus_{n \in \mathbb{N}} V_n$ corresponding to a decomposition of $\pi$ into $\pi'$ weakly-mixing and countably many finite-dimensional and irreducible representations $\pi_n$, such that the cocycle $b_n$ is not reduced in $\overline{H^1}(\pi_n,V_n)$ for some $n$.

\begin{claim}\label{claim}
The restricted cocycle $b_n|_{H^G}$ is not reduced in $\overline{H^1}(\pi_n|_{H^G},V_n)$.
\end{claim}

\begin{proof}[Proof of Claim \ref{claim}]
By the definition of $b$, the cocycle $b_n|_{H^G}$ is not zero. So there exists a large enough $k$ such that $\textrm{supp} (\theta^{(k)}) \cap \textrm{supp} (b_n) \cap H^G$ is non-empty. Using $\theta^{(k)}$-harmonicity of $b_n$ and the invariance under $G$ of $b$, we get $0=\sum_{x \in H^G} b_n(x)\theta^{(k)}(x)$.

Now if $b_n|_{H^G}$ were reduced, by finite dimension it would be a coboundary, giving a vector $v$ in $V_n$ such that $\forall x \in H^G, b_n(x)=v-\pi_n(x)v$. Then
\begin{align*}
0=2\left\langle \sum_{x \in H^G} b_n(x)\theta^{(k)}(x),v\right\rangle=\sum_{x \in H^G}\left\langle b_n(x),v-\pi_n(x)v\right\rangle \theta^{(k)}(x)\\
= \sum_{x \in H^G}|| b_n(x)||^2_{V_n} \theta^{(k)}(x)>0,
\end{align*}
where the middle equality follows classicaly from the cocycle identity.
\end{proof}

By \cite[Proposition 3.2]{Shalom_rigidity}, the claim ensures the existence of $g_0$ in $G$ and a non-zero vector $v$ in $V_n$ which is invariant under $\pi_n|_{H^{G\setminus\{g_0\}}}$. By irreducibility and finite dimension of $\pi_n$, there exist $\gamma_1,\dots,\gamma_r$ in $H \wr G$ such that the vectors $\pi_n(\gamma_i)v$ span $V_n$.

It follows that $\pi_n$ restricted to the intersection $\cap_{i=1}^r  H^{G\setminus\{g_0\}}\gamma_i^{-1}$ fixes $V_n$ pointwise. Moreover as $G$ is infinite, this intersection contains a copy $H^{g_1}$ for some $g_1$ in $G$. By conjugacy, we deduce that $\pi_n|_{H^G}$ is trivial and finally that $b_n|_{H^G}:H^G \rightarrow V_n$ is a non-zero homomorphism. This contradicts our assumption that $H$ has finite abelianization.
\end{proof}

We observe that Claim \ref{claim} is the only new ingredient not appearing in \cite[Section 5.4]{Shalom}.

\bibliographystyle{alpha}
\bibliography{H_FD}

\textsc{\newline J\'er\'emie Brieussel --- Universit\'e de Montpellier 
} --- jeremie.brieussel@umontpellier.fr

\textsc{\newline Tianyi Zheng --- UC San Diego
} --- tzheng2@math.ucsd.edu
\end{document}